\documentclass[11pt,a4paper]{amsart}
\usepackage{amssymb}
\usepackage{geometry}
\usepackage{hyperref}

\date{}

\newtheorem{theorem}{Theorem}[section]
\newtheorem{proposition}[theorem]{Proposition}
\newtheorem{lemma}[theorem]{Lemma}
\newtheorem{corollary}[theorem]{Corollary}
\newtheorem{definition}[theorem]{Definition}
\theoremstyle{remark}
\newtheorem{remark}[theorem]{Remark}
\newcommand{\R}{\mathbb{R}}
\newcommand{\Rn}{\mathbb{R}^n}
\newcommand{\Bilip}{\mathrm{Bilip}}
\newcommand{\PL}{\mathrm{PL}}
\newcommand{\GL}{\mathrm{GL}}
\newcommand{\BS}{\mathrm{BS}}
\newcommand{\QI}{\mathrm{QI}}
\newcommand{\SO}{\mathrm{SO}}
\newcommand{\Diff}{\mathrm{Diff}}
\begin{document}
\title[Orderability and Asymptotic Structure of $\mathrm{QI}(\mathbb{R}^n)$]{Orderability and Asymptotic Structure of $\mathrm{QI}(\mathbb{R}^n)$}
\author{Swarup Bhowmik}
\address{Department of Mathematics, Indian Institute of Science Education and Research Bhopal, Bhopal Bypass Road, Bhauri, Bhopal 462 066, Madhya Pradesh,
India}
\email{swarupbhowmik@iiserb.ac.in}
\author{Deblina Das}
\address{Department of Mathematics, Indian Institute of Technology Palakkad}
\email{212114002@smail.iitpkd.ac.in, deblina099@gmail.com}
\author{Kashyap Rajeevsarathy}
\address{Department of Mathematics, Indian Institute of Science Education and Research Bhopal, Bhopal Bypass Road, Bhauri, Bhopal 462 066, Madhya Pradesh,
India}
\email{kashyap@iiserb.ac.in}
\begin{abstract}
In this article, we study the algebraic and dynamical structure of certain normal subgroups of the quasi-isometry group of Euclidean spaces. We first consider the normal subgroup consisting of quasi-isometries that are asymptotically equal to the identity, and introduce a nested family of normal subgroups that distinguish different orders of sublinear deviation from the identity. We show that the centers of the resulting quotient groups are trivial. We further prove that these quotient groups are neither left-orderable nor locally indicable. We also introduce an asymptotic topology on the quasi-isometry group, yielding a natural metric structure on the quotient and providing a framework for studying large-scale invariants.
\end{abstract}
\thanks{2020 \textit{Mathematics Subject Classification.} 20F60, 20F65, 20F69}
\thanks{\textit{Key words and phrases.} Quasi-isometry group, left-orderable groups.}
\maketitle
\section{Introduction}

The concept of quasi-isometry plays a central role in geometric group theory, serving as a large-scale invariant of metric spaces and finitely generated groups equipped with word metrics. For a metric space $X$, the group of quasi-isometries $\QI(X)$ is defined as the set of equivalence classes of self quasi-isometries under the relation of bounded distance. Although the study of $\QI(X)$ for various classes of spaces and groups has received considerable attention, including irreducible lattices in semisimple Lie groups \cite{Farb}, solvable Baumslag-Solitar groups $\BS(1, n)$ \cite[Theorem 7.1]{Farb_Mosher}, as well as the groups $\BS(m, n)$ where $1 < m < n$ \cite[Theorem 4.3]{Whyte}, the structure of $\QI(\Rn)$ remains only partially understood.

For $n=1$, Sankaran \cite{sankaran} provided an explicit description via the group $\PL_\delta(\R)$, the group of piecewise-linear homeomorphisms of the real line with bounded slopes, revealing a rich algebraic structure containing Thompson's group $F$ and free groups of continuum rank, with trivial center \cite{chakraborty}. Notably, Gromov and Pansu \cite[\textsection 3.3.B]{gromov} demonstrated that, for $n = 1$, $\QI(\mathbb{Z}) \cong \QI(\mathbb{R})$ emerges as an infinite-dimensional group. In \cite{ye2023group}, it was shown that the orientation-preserving quasi-isometry group $\QI^{+}(\R)$ is not simple, admitting a suitable normal subgroup and exhibiting intriguing dynamical properties: it is left-orderable but not locally indicable and has no effective action on $\R$. Later, \cite{bhowmik_chakraborty2} gave an almost complete characterization of $\QI(\R_{+})$, the quasi-isometry group of the positive real line, through a new invariant, revealing that a related quotient remains left-orderable yet not locally indicable.

In contrast, for $n>1$, the groups $\QI(\mathbb{R}^n)$ remain much less understood. Nevertheless, Mitra and Sankaran \cite{mitra_sankaran} showed that $\QI(\mathbb{R}^n)$ contains significant subgroups such as $\Bilip(S^{n-1})$, $\Diff^r(S^{n-1})$, and $\PL(S^{n-1})$. In \cite{bhowmik_chakraborty}, a combinatorial criterion was established for determining when a piecewise-linear homeomorphism of $\mathbb{R}^n$ represents an element of $\QI(\mathbb{R}^n)$, and it was further shown that its center is trivial.

In this article, we study algebraic and dynamical features of the quasi-isometry group $\QI(\Rn)$ from the perspective of large-scale geometry. Despite its fundamental role in geometric group theory, the internal structure of $\QI(\Rn)$, particularly its normal subgroups arising from asymptotic behavior, remains only partially understood. It was shown in \cite{zhaothesis} that $\QI(\Rn)$ is not simple through the existence of a proper normal subgroup consisting of quasi-isometries that are asymptotically close to the identity. However, this result represents only a first step in understanding the asymptotic geometry of $\QI(\Rn)$. In particular, the structure of intermediate normal subgroups and the behavior of central elements in the associated quotient groups have remained largely unexplored.

Motivated by this, we introduce a new family of normal subgroups that capture the asymptotic deviation of quasi-isometries from the identity and provide a framework for analyzing their algebraic and dynamical properties. It is already known from \cite{zhaothesis} that the quasi-isometry group $\QI(\Rn)$ admits a proper normal subgroup
\[H = \left\{[f] \in \QI(\mathbb{R}^n): \lim_{\|x\| \to \infty} \frac{
 \|f(x)-x\|
  }{\|x\|}=0 \right\}\]
  which shows that $\QI(\Rn)$ is not simple.
For each $\alpha\in (0,1)$, we define a family of subsets of $\QI(\Rn)$ as follows:
 $$H_{\alpha}=\{[f]\in \QI(\Rn): \exists~ K,R>0, \|f(x)-x\| \leq K \|x\|^{\alpha}~\text{for~all~} \|x\| \geq R \}.$$ Each \(H_\alpha\) is a normal subgroup of \(\QI(\mathbb{R}^n)\) and satisfies
\(H_\alpha \subset H_\beta \subset H\) whenever \(0<\alpha<\beta<1\).
Our main results are as follows.
\begin{theorem}\label{center_quotient}
The center of the group $\QI(\mathbb{R}^n)/H$ (and hence $\QI(\Rn)/H_\alpha$) is trivial.
\end{theorem}
\begin{theorem}\label{Center_H}
    For each $\alpha\in(0,1)$, the center of the group $H_\alpha$ (and hence $H$) is trivial.
\end{theorem}

In recent years, the notions of left-orderability and locally indicability have attracted considerable attention (see \cite{Deroin_Navas_Rivas,Mann} for instance) due to their deep connections with algebraic,
dynamical, and topological structures. In this work, we examine these properties for the quotient
group \(\QI(\mathbb{R}^n)/H\) and the quotient group $\QI(\mathbb{R}^n)/H_\alpha$, obtaining the following result.
\begin{theorem}\label{left_orderable}
The quotient group \(\QI(\mathbb{R}^n)/H\) (and hence \(\QI(\mathbb{R}^n)/H_\alpha\))
is not left-orderable and therefore not locally indicable.
\end{theorem}
\noindent Furthermore, we observe that the quotient $\QI(\mathbb{R}^n)/H$ remains algebraically rich,
admitting faithful embeddings of several classical groups, including
$\GL(n,\mathbb{R})$ and $\Bilip(S^{n-1})$.

In addition, we introduce an asymptotic topology on $\QI(\mathbb{R}^n)$, defined in terms of relative deviation at infinity. We show that this topology admits a natural pseudo-metric yielding a Hausdorff metric structure on the quotient $\QI(\mathbb{R}^n)/H$. Moreover, we define a continuous invariant capturing asymptotic scaling behavior.

In Section \ref{Preliminaries}, we recall fundamental notions concerning the quasi-isometry group of $\mathbb{R}^n$ and the subgroups $H$ and $H_\alpha$. Section \ref{Properties} investigates the algebraic and dynamical properties of the quotient groups $\QI(\mathbb{R}^n)/H$ and $\QI(\mathbb{R}^n)/H_\alpha$, including the structure of their centers, the presence of torsion elements, and consequences for left-orderability. Section \ref{Discussion} introduces an asymptotic topology on $\QI(\mathbb{R}^n)$ and studies its relation to the subgroup $H$, yielding a natural metric structure on the quotient and associated large-scale invariants.
\section{Preliminaries}\label{Preliminaries}
In this section, we collect the basic notions and conventions used throughout the article. We begin with the definition of quasi–isometries, which is the key framework for our study. We then recall standard definitions from the theory of orderable groups, and introduce several pieces of notation used in later sections.
\subsection*{Quasi–isometries}

Let $(X,d_1)$ and $(X',d_2)$ be metric spaces. A map $f:X\to X'$ (not necessarily
continuous) is a \emph{quasi–isometric embedding} if there exist constants $K>1$ and
$C>0$ such that
$$\frac1K\, d_1(x_1,x_2)-C \;\le\; d_2\!\big(f(x_1),f(x_2)\big)
\;\le\; K\, d_1(x_1,x_2)+C$$
for all $x_1,x_2\in X$. If, in addition, every point of $X'$ lies within a uniformly bounded distance of $f(X)$, then $f$ is called a \emph{quasi–isometry}. Two maps $f,g:X\to X$ are \emph{quasi–isometrically equivalent} if there exists
$\alpha>0$ such that
$
d_1(f(x),g(x))<\alpha \text{ for all } x\in X.
$
The equivalence class of a quasi–isometry $f$ is denoted $[f]$.
The set $\QI(X)=\{[f]\mid f:X\to X \text{ is a quasi–isometry}\}$
forms a group with respect to composition: $[f]\circ[g]=[fg]$.
If $f:X\to X'$ is a quasi–isometry, then it induces an isomorphism $\QI(X)\cong \QI(X')$. For convenience, we sometimes write $f$ instead of $[f]$ when no confusion can arise. For a general introduction to quasi-isometries and their role in geometric group theory, we refer the reader to \cite{loh}. 
\subsection*{\texorpdfstring{The subgroups {\boldmath $H$} and {\boldmath $H_\alpha$}}{The subgroups H and H-alpha}}

Following \cite{zhaothesis}, we consider the normal subgroup
\[
H=\left\{
[f]\in \QI(\mathbb{R}^n):
\lim_{\|x\|\to\infty}\frac{\|f(x)-x\|}{\|x\|}=0
\right\},
\]
consisting of quasi–isometries asymptotically close to the identity.

For $0<\alpha<1$, we define
\[
H_\alpha=\Big\{
[f]\in \QI(\mathbb{R}^n):
\exists\,K,R>0 \text{ with }
\|f(x)-x\|\le K\|x\|^\alpha \text{ for all }\|x\|\ge R
\Big\}.
\]

These families of subgroups form the central objects of our algebraic and
dynamical analysis in later sections.
\subsection*{Centers and centralizers}
Let $G$ be a group. The center of $G$ and the centralizer of $g\in G$ are defined by
\[
Z(G)=\{g\in G : gh=hg \text{ for all } h\in G\}, \quad C_G(g)=\{h\in G : hg=gh\},
\]
respectively.
\subsection*{Left–orderable and locally indicable groups}
\begin{definition}
A group $G$ is \emph{left–orderable} if it admits a total order $\le$ such that
\[
g\le h \ \Longrightarrow\ fg\le fh \qquad\text{for all } f\in G.
\]
\end{definition}
Subgroups of left–orderable groups are also left–orderable. Typical examples include
torsion-free abelian groups, free groups, and the group of germs at $\infty$ of homeomorphisms of $\mathbb{R}$ \cite{Mann}.

An element $g\in G$ is a \emph{torsion element} if $g^m=e$ for some integer $m>1$.
Since left–orderable groups are torsion–free, the existence of torsion in the quotient
groups above is a key ingredient in Section~3.
\begin{definition}
A group $G$ is \emph{locally indicable} if every finitely generated nontrivial subgroup
admits a nontrivial homomorphism onto $\mathbb{Z}$.
\end{definition}

A classical theorem of Brodskii and Howie (see Deroin–Navas–Rivas~\cite{Deroin_Navas_Rivas}) asserts that torsion–free one–relator groups are locally indicable. Knot groups also lie in this class, as do the braid groups $B_3$ and $B_4$, while $B_n$ fails to be locally indicable for $n\ge5$.  
Local indicability implies left–orderability; this implication will be used later when
examining torsion in the quotient groups $\QI(\mathbb{R}^n)/H$ and $\QI(\mathbb{R}^n)/H_\alpha$.

\section[Algebraic and dynamical properties of quotient groups]{Algebraic and dynamical properties of quotient groups of \texorpdfstring{$\QI(\mathbb{R}^n)$}{\QI(Rn)}}\label{Properties}

In this section, we first show that the group \(\QI(\mathbb{R}^n)\) is not simple by providing a family of normal subgroups \(H_\alpha\) for each \(\alpha \in (0,1)\). We then establish that the centers of the quotient groups \(\QI(\mathbb{R}^n)/H\) and \(\QI(\mathbb{R}^n)/H_\alpha\) are trivial. Furthermore, we show that the centralizers of \(H_\alpha\) and \(H\) are also trivial. In addition, we examine the centralizers of certain elements of \(\QI(\mathbb{R}^n)\) and conclude the section with a discussion on the torsion-freeness and left-orderability of the corresponding quotient groups.
\begin{theorem}
     For each $\alpha\in (0,1)$, $H_{\alpha}$ is a nontrivial normal subgroup of $\QI(\Rn).$
\end{theorem}
\begin{proof}

     Let $[f],[g]\in H_\alpha$. Then there exist $K_f,K_g,R_f,R_g>0$ such that
     \[\|f(x)-x\|\leq K_f\|x\|^{\alpha}\quad \text{for all }\|x\|\geq R_f\quad \text{and}\quad \|g(x)-x\|\leq K_g\|x\|^{\alpha}\quad \text{for all }\|x\|\geq R_g.\]
     Now,
     $\|fg(x)-x\|\leq \|f(g(x))-g(x)\|+\|g(x)-x\|.$
     Let $R_{f\circ g}=\max\{R_f,R_g\}$. Then
     \[\|fg(x)-x\|\leq K_f\|g(x)\|^{\alpha}+K_g\|x\|^{\alpha}\quad \text{for all $\|x\|\geq R_{f\circ g}$}. \]
     Since $g$ is a quasi-isometry, there exists a constant $K'>1$ such that $\|g(x)\|^{\alpha}\leq K'^{\alpha}\|x\|^{\alpha}$.
        Let $K_{f\circ g}=\max\{K'^{\alpha}K_f,K_g\}$. Then from above, for all $\|x\|\geq R_{f\circ g}$, $\|fg(x)-x\|\leq K_{f\circ g}\|x\|^{\alpha}.$
     This shows that $[f g]\in H_{\alpha}$.

Similarly, if $[f]\in H_\alpha$, then using that $f^{-1}$ is a quasi-isometry and $ff^{-1}$ is at bounded distance from the identity, one obtains
$
\|f^{-1}(x)-x\|\le C\|x\|^\alpha
$
for large $\|x\|$, showing $[f^{-1}]\in H_\alpha$. Therefore, $H_\alpha$ is a subgroup of $\QI(\mathbb{R}^n)$.

     Now, we show that $H_\alpha$ is normal in $\QI(\Rn)$. Let $[f]\in \QI(\Rn)$ and $[g]\in H_\alpha$. Then there exist $K_g, R_g>0$ such that $\|g(x)-x\|\leq K_g\|x\|^\alpha$ for all $\|x\|\geq R_g$. Then,
\begin{align*}
         \|fgf^{-1}(x)-x\|&=\|f(g(f^{-1}(x)))-f(f^{-1}(x))+f(f^{-1}(x))-x\|\\&\leq \|f(g(f^{-1}(x)))-f(f^{-1}(x))\|+\|f(f^{-1}(x))-x\|\\&\leq \lambda\|gf^{-1}(x)-f^{-1}(x)\|+ C+ \mu,
\end{align*}
   since $f$ is a quasi-isometry and $ff^{-1}$ is quasi-isometrically equivalent to the identity map with $\lambda>1$ and $C,\mu>0$.
   Again, since $[g]\in H_\alpha$, $\|gf^{-1}(x)-f^{-1}(x)\|\leq K_g\|f^{-1}(x)\|^{\alpha}$. Since $f^{-1}$ is a quasi-isometry, $\|f^{-1}(x)\|\leq a\|x\|+b$ for some $a>1$ and $b>0$. Then from above, we get, for all $\|x\|\geq R_g$,
\begin{align*}
       \|fgf^{-1}(x)-x\|&\leq \lambda K_g(a\|x\|+b)^{\alpha}+C+\mu\\&\leq \lambda K_g(a^\alpha\|x\|^\alpha+b^\alpha)+C+\mu \leq K'\|x\|^\alpha,\quad \text{for some }K'>0.
\end{align*}
   Therefore, $[f]\circ [g]\circ [f]^{-1}\in H_\alpha$.

Now we show that there is at least one non-identity element in $H_\alpha$.
  For example, let $f: \Rn \to \Rn$ be given by
   $f(x)=x+A\ln{(1+\|x\|)}v$
   where $v$ is a unit vector of $\Rn$ and $A$ is a real constant with $|A|<1$. It is straightforward to check that $[f]\in \QI(\Rn)$.
   For instance, note that
\[
| \ln(1+\|x\|)-\ln(1+\|y\|)|\le \|x-y\|, \text{ and therefore }
\]
\[
\|f(x)-f(y)\|=\|(x-y)+A(\ln(1+\|x\|)-\ln(1+\|y\|))v\|\le(1+|A|)\|x-y\|.
\]
A similar estimate gives a coarse linear lower bound and quasi-surjectivity.

Let $x_n=nu$ where $u=(1,0,0,\ldots, 0)$ be a sequence in $\Rn$. Now,
\[
\|f(x_n)-x_n\|=\|A\ln{(1+n)}\|\to \infty \text{ as } n \to \infty.
\]
Hence, $[f]\neq [id]$. Also,
  $$\frac{\|f(x)-x\|}{\|x\|^{\alpha}}=\frac{|A|\ln(1+\|x\|)}{\|x\|^\alpha}\leq |A|M_\alpha,$$ where $M_\alpha$ is a constant depending on $\alpha$. Such a constant $M_\alpha$ always exists since $h(r)=\displaystyle\frac{\ln{(1+r)}}{r^\alpha}$ is continuous on $(0, \infty)$ and both limits $\displaystyle\lim_{r\to \infty} h(r)$ and $\displaystyle\lim_{r\to 0^+} h(r)$ exist finitely. Therefore, $[f]\in H_\alpha$.
\end{proof}
\begin{corollary}
The group $\QI(\Rn)$ is not simple. Indeed, by the preceding theorem, it admits a family of nontrivial normal subgroups $H_\alpha$.
\end{corollary}
\begin{remark}
      Note that each $H_{\alpha}$ is a subgroup of $H$. For instance, let $[f]\in H_\alpha$, i.e., there exist $K_f, R_f>0$ such that $\|f(x)-x\|\leq K_f\|x\|^\alpha$ for all $\|x\| \geq R_f$. Since $\alpha-1<0$,
     $$\frac{\|f(x)-x\|}{\|x\|}\leq K_f\|x\|^{\alpha-1}\to 0 \text{ as } \|x\| \to\infty.$$
      This shows that $[f]\in H$ for each $\alpha\in (0,1).$
\end{remark}
\begin{remark}
      For $0<\alpha <\beta <1$, $H_{\alpha}\subseteq H_{\beta}$.

      For instance, let $[f]\in H_\alpha$, that is, there exist $K_f, R_f>0$ such that $\|f(x)-x\|\leq K_f\|x\|^\alpha$ for all $\|x\| \geq R_f$. Let $R'_f=\max\{1,R_f\}$. Thus, for $\|x\|\geq R'_f$, $\|x\|\geq 1$. Now
 $\|f(x)-x\|\leq K_f\|x\|^\alpha\leq K_f \|x\|^\beta \text{ for all $\|x\|\geq R'_f.$}$
 Therefore, $[f]\in H_\beta.$
\end{remark}
\begin{remark}

We observe that $\displaystyle\bigcup_{0<\alpha<1} H_{\alpha}$ is a proper subgroup of $H$.
For instance, let $f:\mathbb{R}^n \to \mathbb{R}^n$ be defined by $f(x)=x+\displaystyle\frac{x}{\ln(2+\|x\|)}$.
Since $\displaystyle\frac{\|f(x)-x\|}{\|x\|}=\displaystyle\frac{1}{\ln(2+\|x\|)}\!\to\!0$ as $\|x\|\!\to\!\infty$, we have $[f]\in H$.
However, for any $0<\alpha<1$ and for any $K>0$, $\|f(x)-x\|=\displaystyle\frac{\|x\|}{\ln(2+\|x\|)}\not\le K\|x\|^{\alpha}$ for large $\|x\|$, as $\displaystyle\frac{\|x\|^{1-\alpha}}{\ln(2+\|x\|)}\!\to\!\infty$, hence $[f]\notin H_{\alpha}$.
\end{remark}

To prove Theorem \ref{center_quotient}, we need the following result, which gives an equivalent statement about equality of two cosets in $\QI(\Rn)/H$.
\begin{lemma}\label{Lemma Center}
Let $[f], [g] \in \QI(\mathbb{R}^n)$. Then $[f]$ and $[g]$ lie in the same coset of $H$, i.e., $H[f]= H[g]$,
if and only if for every $\epsilon>0$ there exists $M>0$ such that
\[
\frac{\|f(x)-g(x)\|}{\|x\|}<\epsilon \quad \text{for all } \|x\|>M.
\]
\end{lemma}
\begin{proof}
Suppose first that $[f]$ and $[g]$ belong to the same coset of $H$. Then there exists $[h]\in H$ such that
$f = h\circ g$. Hence, for all $x\in \mathbb{R}^n$,
\begin{align}\label{eq:equiv_condition}
\frac{\|f(x)-g(x)\|}{\|x\|}
= \frac{\|h(g(x))-g(x)\|}{\|g(x)\|}\cdot \frac{\|g(x)\|}{\|x\|}.
\end{align}
Since $[g]\in \QI(\mathbb{R}^n)$, the ratio $\|g(x)\|/\|x\|$ is bounded and $\|g(x)\|\to \infty$ as $\|x\|\to \infty$. Again, since $h\in H$, we have
\[
\lim_{\|x\|\to\infty}\frac{\|h(g(x))-g(x)\|}{\|g(x)\|}=0,
\text{ and therefore }
\lim_{\|x\|\to\infty}\frac{\|f(x)-g(x)\|}{\|x\|}=0 \text{ (by \eqref{eq:equiv_condition})}.
\]

  Conversely, assume the condition holds.
Define $[h'] := [fg^{-1}]$, so that $f = h'g$. We claim that $[h']\in H$. Indeed, for large $\|x\|$,
\begin{align}\label{eq:converse}
\frac{\|h'(g(x))-g(x)\|}{\|g(x)\|}\cdot \frac{\|g(x)\|}{\|x\|} < \epsilon.
\end{align}
Since $g$ is a quasi-isometry, there exists $M'>1$ such that, for all sufficiently large $\|x\|$, $\displaystyle\frac{1}{M'} < \displaystyle\frac{\|g(x)\|}{\|x\|}.$
Applying this to \eqref{eq:converse},
$$\frac{\|h'(g(x))-g(x)\|}{\|g(x)\|} < M'\epsilon.$$
Letting $y=g(x)$, and noting that $\|x\|\to\infty$ implies $\|y\|\to\infty$, we conclude that
\[
\lim_{\|y\|\to\infty}\frac{\|h'(y)-y\|}{\|y\|}=0.
\]
Therefore, $[h']\in H$. This completes the proof.
\end{proof}
\begin{remark}
As a consequence of Lemma \ref{Lemma Center}, the subgroup $H$ has infinite index in $\QI(\mathbb{R}^n)$. Indeed, for $\lambda>0$, consider the dilation $D_\lambda(x)=\lambda x$. Then
\[
\frac{\|D_\lambda(x)-D_\mu(x)\|}{\|x\|}=|\lambda-\mu|,
\]
so $[D_\lambda]$ and $[D_\mu]$ lie in the same coset of $H$ if and only if $\lambda=\mu$. Thus $\{[D_\lambda]:\lambda>0\}$ determines uncountably many distinct cosets in $\QI(\mathbb{R}^n)/H$.
\end{remark}

The preceding lemma provides a necessary and sufficient condition for two quasi-isometries $f$ and $g$ to commute modulo $H$. We now proceed to determine the structure of the center of $\QI(\Rn)/H$. In the proof of the following theorem, we use a technique similar to that in \cite{bhowmik_chakraborty}.
\subsection{Proof of Theorem \ref{center_quotient}}

 Let $H[f]$ be a non-identity element of $\QI(\mathbb{R}^n)/H$, that is, $[f]\notin H$.
Then there exists a sequence $\{x_n\}\subset\mathbb{R}^n$ with $\|x_n\|\to\infty$ such that
\[
\lim_{\|x_n\|\to\infty} \frac{\|f(x_n)-x_n\|}{\|x_n\|} > 0.
\]
Hence, we can find a subsequence $\{x_m\}$ and $\epsilon>0$ such that
$\|f(x_m)-x_m\| \;\geq\; \epsilon \|x_m\| \quad \text{for all } m.$
Since
$\|x_m\|,
\|f(x_m)\|,
\|f(x_m)-x_m\|
\longrightarrow \infty$,
there exists a subsequence $\{a_m\}$ of $\{x_m\}$ such that
$
\|a_m\|,
\|f(a_m)\|,
\|f(a_m)-a_m\|
$
are strictly increasing sequences tending to infinity and satisfying
$\|a_{m+1}\|>\|f(a_m)\|$
for all $m$.

Define integers
$m_j = \Big\lceil \tfrac{\epsilon}{2}\|a_j\|\Big\rceil.$
Then $m_j\geq \frac{\epsilon}{2}\|a_j\| \text{ for all } j$.
Passing to a subsequence $\{m_{j_k}\}$, we may assume $\{m_{j_k}\}$ is strictly increasing. We can construct a subsequence $\{a_{j_k}\}$ of $\{a_j\}$ such that
for each $k$,
\[
r_{m_{j_k}} := \min\!\Big\{ m_{j_k},\, \tfrac12\|f(a_{j_k})-a_{j_k}\|\Big\},
\]
is a strictly increasing sequence, and the closed balls $\overline{D_k}$ of radius $r_{m_{j_k}}$ centered at $f(a_{j_k})$
are pairwise disjoint and contain none of the points $a_{j_i},f(a_{j_i})$ for $i\neq k$. This can be checked by induction. We
define $g:\mathbb{R}^n\to\mathbb{R}^n$ by
\[
g(x) =
\begin{cases}
(\rho_k^{-1}\circ h \circ \rho_k)(x), & x\in \overline{D_k},\\
x, & x\notin \bigcup_k \overline{D_k},
\end{cases}
\]
 where $\displaystyle\rho_k(x) = \frac{x - f(a_{j_k})}{r_{m_{j_k}}}$ for $x\in\overline{D_k}$; this is the rescaling to the closed unit disk $\overline{\mathbb{D}^1_n}$, and
$h:\overline{\mathbb{D}^1_n}\to \overline{\mathbb{D}^1_n}$ is a quasi-isometry with
$h(0,\dots,0)=(0,\dots,0,\tfrac14)$.
By the results in \cite{bhowmik_chakraborty, mitra_sankaran}, such a $g$ is a quasi-isometry.

\noindent{Now we show that $[g]\notin H$.}
At the points $f(a_{j_k})$,
\begin{align}\label{Not in H}
    \frac{\|g(f(a_{j_k}))-f(a_{j_k})\|}{\|f(a_{j_k})\|}=\frac{\|f(a_{j_k})+(0,0,\ldots, \frac{r_{m_{j_k}}}{4})-f(a_{j_k})\|}{\|f(a_{j_k})\|}=\frac{r_{m_{j_k}}}{4\|f(a_{j_k})\|}.
\end{align}
By construction $r_{m_{j_k}} \geq \tfrac{\epsilon}{2}\|a_{j_k}\|$.
Since $f$ is a quasi-isometry, the ratio $\|f(a_{j_k})\|/\|a_{j_k}\|$ is bounded above and below away from $0$.
Thus,
\[
\limsup_{k\to\infty} \frac{\|g(f(a_{j_k}))-f(a_{j_k})\|}{\|f(a_{j_k})\|} > 0, \text{ and therefore } [g]\notin H.
\]
\noindent{Now,}
at the points $a_{j_k}$ we compute
\[\frac{\|(f\circ g) (a_{j_k})-(g\circ f)(a_{j_k})\|}{\|a_{j_k}\|}=\frac{\|(0,0,\ldots, \frac{r_{m_{j_k}}}{4})\|}{\|a_{j_k}\|}=\frac{r_{m_{j_k}}}{4\|a_{j_k}\|}\geq \frac{\epsilon}{8}.\]
By Lemma~\ref{Lemma Center}, this implies
$H[f\circ g]\neq H[g\circ f]$,
i.e.,\ $H[f]H[g]\neq H[g]H[f]$.

\noindent Therefore, the center of the quasi-isometry group $\QI(\mathbb{R}^n)/H$ is trivial.

Since $H_\alpha\subset H$ for each $\alpha\in(0,1)$, any nontrivial element $[f]\in \QI(\mathbb{R}^n)/H$ necessarily lies outside $H_\alpha$ as well. In the proof of the preceding theorem, we construct a quasi-isometry $[g]\notin H$ associated with such an element $[f]$, and hence $[g]\notin H_\alpha$. This immediately implies that the center of the quotient group $\QI(\mathbb{R}^n)/H_\alpha$ is trivial.
\hspace{8cm} \qedsymbol
\begin{remark}

A slight modification of the proof of Theorem~\ref{center_quotient} also shows that, for each
$0<\alpha<1$, the center of the quotient group
$
H/H_\alpha
$
is trivial.

Indeed, let $H_\alpha[f]$ be a nontrivial element of $H/H_\alpha$, so that
$[f]\in H\setminus H_\alpha$. Then there exists a sequence
$\{a_m\}\subset \mathbb{R}^n$ with $\|a_m\|\to\infty$ and a constant
$\varepsilon>0$ such that
\[
\|f(a_m)-a_m\|\ge \varepsilon \|a_m\|^\alpha
\quad \text{for all } m.
\]
Choose $\beta$ such that $\alpha<\beta<1$ and define
$r_m:=\min\left\{\|a_m\|^\beta,\,
\frac{1}{2}\|f(a_m)-a_m\|\right\}.$

Repeating the construction of Theorem~\ref{center_quotient} with closed balls centered at
$f(a_m)$ of radius $r_m$, we obtain a quasi-isometry $[g]$ such that
$[g]\in H$, since
\[
\frac{\|g(x)-x\|}{\|x\|}
\to 0
\quad \text{as } \|x\|\to\infty,
\]
which follows from the choice $\beta<1$. Furthermore, $[g]\notin H_\alpha$. Indeed,
\[
\frac{\|g(f(a_m))-f(a_m)\|}{\|f(a_m)\|^\alpha}
=
\frac{\|f(a_m)-a_m\|}{8\|f(a_m)\|^\alpha}.
\]
Since $[f]\notin H_\alpha$, it follows that
\[
\frac{\|f(a_m)-a_m\|}{\|a_m\|^\alpha}\to\infty,
\text{ which implies }\]
\[ \frac{\|g(f(a_m))-f(a_m)\|}{\|f(a_m)\|^\alpha}\to\infty.\]
On the other hand,
\[
\|(f\circ g)(a_m)-(g\circ f)(a_m)\|
=
\frac{r_m}{4}, \text{ which implies }\] 
\[\frac{\|(f\circ g)(a_m)-(g\circ f)(a_m)\|}{\|a_m\|^\alpha}
=
\frac{r_m}{4\|a_m\|^\alpha}
\to\infty,
\]

since $\beta>\alpha$.
Therefore,
$H_\alpha[f\circ g]\neq H_\alpha[g\circ f].$
Consequently,
$Z(H/H_\alpha)=\{e\}$.
\end{remark}

From the theorem above, the centers of both quotient groups $\QI(\mathbb{R}^n)/H$ and $\QI(\mathbb{R}^n)/H_\alpha$ are trivial. To further understand commutation phenomena in $\QI(\mathbb{R}^n)$, it is natural to study the centralizers of $H$ and $H_\alpha$ themselves. We denote these by
\begin{align*}
    C(H)=&\{[f]\in \QI(\mathbb{R}^n): [f]\circ[g]=[g]\circ[f]\ \text{for all }[g]\in H\} \text{ and }\\
C(H_\alpha)=&\{[f]\in \QI(\mathbb{R}^n): [f]\circ[g]=[g]\circ[f]\ \text{for all }[g]\in H_\alpha\}.
\end{align*}
\begin{lemma}
   The centralizer of $H_\alpha$ in $\QI(\Rn)$ is trivial.
\end{lemma}
\begin{proof}

     Suppose $[f]$ is a non-identity element of $\QI(\mathbb{R}^n)$. Then there exists a sequence
$\{a_m\}$ in $\mathbb{R}^n$ such that $\|a_m\|$, $\|f(a_m)\|$, and $\|f(a_m)-a_m\|$ are monotonically increasing sequences converging to infinity, with the additional property that $\|a_{m+1}\| > \|f(a_m)\|$ for each $m$.
Since $f$ is a quasi-isometry of $\Rn$, there exist $K>1$ and $C>0$ such that for all $x_1,x_2\in\Rn$,
\begin{center}
		$\displaystyle\frac {1}{K}\|x_1-x_2\|-C\leq \|f(x_1)-f(x_2)\|\leq K\|x_1-x_2\|+C$.
\end{center}
From this sequence $\{a_m\}$, we can extract a subsequence $\{b_m\}$ together with disjoint balls $\overline{D_m}$ centered at $f(b_m)$ of radii $r_m$, satisfying:
\begin{enumerate}
    \item $\{r_m\}$ is strictly increasing and diverges to infinity,
    \item each $\overline{D_m}$ contains no point of $\{b_j\}$ and no $f(b_j)$ other than $f(b_m)$, and
    \item $\overline{D_p} \cap \overline{D_q} = \varnothing$ for $p \neq q$,
\end{enumerate}
with
\[
r_m = \min\left\{\displaystyle\lfloor \frac{{\|b_m\|}^{\frac {\alpha}{2}}}{K+1}\rfloor, \, \tfrac{1}{2}\|f(b_m)-b_m\|\right\}.
\]
Next, we construct an element $[g] \in H_\alpha$ (in the same spirit as in the proof of Theorem~\ref{center_quotient}) and show that $[f]$ and $[g]$ do not commute.

First, if $x \notin \bigcup_m \overline{D_m}$, then $g(x)=x$. If $x \in \overline{D_m}$ for some $m$, then $g(x)\in \overline{D_m}$, and hence
\[
\|g(x)-x\| \leq C' r_m,
\text{ for some uniform constant }C'.\]
Since $f$ is a quasi-isometry,
\[
\frac{1}{K}\|b_m\|-C \leq \|f(b_m)\| \leq K\|b_m\|+C.
\]
Moreover, for any $x \in \overline{D_m}$, $\bigl|\|x\|-\|f(b_m)\|\bigr| \leq \|x-f(b_m)\| \leq r_m.$ 

\noindent This implies
$
-r_m + \|f(b_m)\| \leq \|x\| \leq r_m + \|f(b_m)\|.$

\noindent Since $r_m \leq \displaystyle\frac {{\|b_m\|}^{\frac {\alpha}{2}}}{K+1} \leq \displaystyle\frac {\|b_m\|}{K+1}$, it follows that
\[
-\displaystyle\frac {\|b_m\|}{K+1}+\frac{1}{K}\|b_m\|-C \leq \|x\| \leq \displaystyle\frac {\|b_m\|}{K+1}+K\|b_m\|+C.
\]
Consequently, for large $m$
\[
\frac{\|g(x)-x\|}{\|x\|^{\alpha}} \leq \frac{C'r_m}{\left(\frac {\|b_m\|}{K(K+1)}-C\right)^{\alpha}} \leq \frac{\frac {C'{\|b_m\|^{\frac {\alpha}{2}}}}{K+1}}{\left(\frac {\|b_m\|}{K(K+1)}-C\right)^{\alpha}} \longrightarrow 0 \quad \text{as } m\to\infty.
\]
Therefore, in all cases,
 $[g] \in H_\alpha$.
 Finally, consider the sequence $\{b_m\}$. By construction of $g$, we obtain
\[
\|(f\circ g)(b_m)-(g\circ f)(b_m)\| = \Big\|\big(0,0,\ldots,\tfrac{r_m}{4}\big)\Big\| \to \infty
\quad \text{as } m\to\infty.
\]
Hence, $[f]\circ [g] \neq [g]\circ [f]$. This shows that the centralizer $C(H_\alpha)$ of $H_\alpha$ is trivial.
\end{proof}
 We are now ready to prove Theorem \ref{Center_H}.
\subsection{Proof of Theorem \ref{Center_H}}
    From the lemma above, we conclude that the center of $H_\alpha$, $Z(H_\alpha)$, is trivial, since $Z(H_\alpha)=C(H_\alpha)\cap H_\alpha$ and $C(H_\alpha)$ is trivial.
    Since $H_\alpha$ is a subgroup of $H$ for each $\alpha\in(0,1)$, the lemma above also shows that the centralizer of $H$ is trivial: if $[g]\in H_\alpha$, then $[g]\in H$, and $[g]$ does not commute with the nontrivial element $[f]$ of $\QI(\Rn)$. Therefore, the center of $H$, $Z(H)$, is also trivial, since $Z(H)=C(H)\cap H$.
    \hspace{9cm}\qedsymbol

In group theory, torsion plays a crucial role, particularly in the study of geometric groups and their structure. A key algebraic result of this paper is the existence of nontrivial torsion elements in the quotient $\QI(\Rn)/H$.
\begin{lemma}
For each $n\in\mathbb{N}$, the quotient group $\QI(\mathbb{R}^n)/H$ has a torsion element.
\end{lemma}
\begin{proof}
If $n=1$, take the reflection $r:\mathbb{R}\to\mathbb{R}$, $r(x)=-x$.
Then $r$ is an isometry, $r^2=\mathrm{id}$, and
\[
\frac{\|r(x)-x\|}{\|x\|}=\frac{2\|x\|}{\|x\|}=2\nrightarrow0,
\]
so $[r]\notin H$. Hence $[r]H$ is a nontrivial element of order $2$ in $\QI(\mathbb{R})/H$.

Now let $n\ge2$ and fix an integer $k\ge2$. Let $R_{2\pi/k}\in \SO(2)$ be the rotation by angle $2\pi/k$ in a chosen $2$-plane of $\mathbb{R}^n$, and let
\[
\widetilde R=\begin{pmatrix} R_{2\pi/k} & 0\\[4pt] 0 & I_{n-2}\end{pmatrix}\in \mathrm{O}(n)
\]
be the block-diagonal extension acting as the identity on the orthogonal complement. Since $\widetilde R$ is an isometry, we have $[\widetilde R]\in \QI(\mathbb{R}^n)$ and $\widetilde R^k=\mathrm{id}$, so $[\widetilde R]^k=[\mathrm{id}]$.

It remains to show $[\widetilde R]\notin H$. Write a point $x\in\mathbb{R}^n$ as $x=(v,w)$ with $v$ in the chosen $2$-plane and $w$ in its orthogonal complement. Using the standard chord formula for a rotation by angle $\theta=2\pi/k$,
\begin{align*}
\|\widetilde R(x)-x\|^2&=\|R_{2\pi/k}(v)-v\|^2\\&=\|R_{2\pi/k}(v)\|^2+\|v\|^2-2\|R_{2\pi/k}(v)\|\|v\| \cos\!\Big(\frac{2\pi}{k}\Big)\\&=4\|v\|^2\sin^2\Big(\frac{\pi}{k}\Big).
\end{align*}
Hence, for any $x=(v,w)$,
\[
0\le \frac{\|\widetilde R(x)-x\|}{\|x\|}=\frac{2\|v\|\sin(\pi/k)}{\sqrt{\|v\|^2+\|w\|^2}}\le 2\sin\!\Big(\frac{\pi}{k}\Big).
\]
Taking the sequence $x_m=(v_m,0)$ with $\|v_m\|\to\infty$ shows
\[
\limsup_{\|x\|\to\infty}\frac{\|\widetilde R(x)-x\|}{\|x\|}=2\sin\!\Big(\frac{\pi}{k}\Big)>0,
\]
while taking $x_m=(v_0,w_m)$ with fixed nonzero $v_0$ and $\|w_m\|\to\infty$ gives
\[
\liminf_{\|x\|\to\infty}\frac{\|\widetilde R(x)-x\|}{\|x\|}=0.
\]
In particular, the (ordinary) limit $\displaystyle\lim_{\|x\|\to\infty}\displaystyle\frac{\|\widetilde R(x)-x\|}{\|x\|}$ does not vanish, hence $[\widetilde R]\notin H$. Therefore the coset $[\widetilde R]H$ is a nontrivial torsion element of order dividing $k$ in $\QI(\mathbb{R}^n)/H$. Since the same argument applies to $\widetilde R^m$ for $1\le m<k$ (whose corresponding angles $2\pi m/k$ are nonzero), the order of $[\widetilde R]H$ in the quotient is exactly $k$.

Thus, for every $n\ge1$ the quotient $\QI(\mathbb{R}^n)/H$ contains a nontrivial torsion element (order $2$ when $n=1$, and elements of any finite order $k\ge2$ when $n\ge2$).
\end{proof}
\begin{corollary}
For each $n\in\mathbb{N}$ and every $0<\alpha<1$, the quotient group $\QI(\mathbb{R}^n)/H_\alpha$ contains nontrivial torsion elements.
\end{corollary}
\begin{proof}
Since $H_\alpha\subset H$, any class $[g]\notin H$ is also not in $H_\alpha$.
From the proof of the previous lemma there exists, for each integer $k\ge2$, an isometry $\widetilde R\in \mathrm{O}(n)$ with $\widetilde R^k=\mathrm{id}$ but $\widetilde R^m\notin H$ for $1\le m<k$.
Thus, $\widetilde R^m\notin H_\alpha$ for $1\le m<k$, while $\widetilde R^k=\mathrm{id}\in H_\alpha$.
Therefore, the coset $[\widetilde R]H_\alpha$ is a nontrivial element of exact order $k$ in $\QI(\mathbb{R}^n)/H_\alpha$.
The $n=1$ case is analogous using the reflection of order $2$. Hence $\QI(\mathbb{R}^n)/H_\alpha$ is not torsion-free.
\end{proof}

We are now ready to prove Theorem \ref{left_orderable}.
\subsection{Proof of Theorem \ref{left_orderable}}
Since any left-orderable group is torsion-free, the lemma above and the corresponding corollary immediately imply that the quotient groups $\QI(\Rn)/H$ and $\QI(\Rn)/H_\alpha$ are not left-orderable. It is well known that every locally indicable group is left-orderable, so the groups $\QI(\Rn)/H$ and $\QI(\Rn)/H_\alpha$ are not locally indicable.
\hfill\qedsymbol
\begin{remark}
When $n=1$, every element $[f]\in H$ preserves the two ends of $\mathbb{R}$. Indeed, if
\[
\lim_{|x|\to\infty}\frac{|f(x)-x|}{|x|}=0,
\]
then, for sufficiently large $|x|$,
\[
|f(x)-x|<\frac{|x|}{2},
\]
which implies $f(x)\to\infty$ whenever $x\to\infty$ and $f(x)\to-\infty$ whenever $x\to-\infty$. Hence
$H\subseteq \QI^{+}(\mathbb{R})$. Since $\QI^{+}(\mathbb{R})$ is torsion-free by \cite[Lemma~3.1]{sankaran}, it follows that $H$ is torsion-free when $n=1$.
For $n\ge2$, the existence of nontrivial torsion elements in $H$ appears to be an open question.
\end{remark}
\begin{remark}
The embeddings constructed by Mitra and Sankaran \cite[Theorem 1.1]{mitra_sankaran} also induce faithful embeddings of several natural groups into the quotient $\QI(\mathbb{R}^n)/H$.
Indeed, let $\phi\in \Bilip(S^{n-1})$ and let $\widetilde{\phi}$ denote its radial extension to $\mathbb{R}^n$. If $\phi\neq id$, then there exists $u\in S^{n-1}$ such that $\phi(u)\neq u$. Hence,
\[
\frac{\|\widetilde{\phi}(ru)-ru\|}{\|ru\|}
=
\|\phi(u)-u\|>0
\]
for every $r>0$, so $[\widetilde{\phi}]\notin H$. Consequently, the embedding $\Bilip(S^{n-1})\hookrightarrow \QI(\mathbb{R}^n)$
constructed in \cite{mitra_sankaran} descends to a faithful embedding $\Bilip(S^{n-1})\hookrightarrow \QI(\mathbb{R}^n)/H.$
The same argument applies to the subgroups $\Diff^{\,r}(S^{n-1})$ and $\PL(S^{n-1})$. Since every element of $\GL(n,\mathbb{R})$ is bi-Lipschitz, there is a natural homomorphism $\GL(n,\mathbb{R})\longrightarrow \QI(\mathbb{R}^n)/H.$
It is immediate from the definition of $H$ that this homomorphism is injective and hence embeds faithfully into $\QI(\mathbb{R}^n)/H$.
\end{remark}

The triviality of the centralizers of $H$ and $H_\alpha$ implies that no nontrivial element of $\QI(\mathbb{R}^n)$ commutes with all elements of these subgroups. We now examine centralizers of specific elements of $\QI(\mathbb{R}^n)$, focusing on geometrically natural examples.
\begin{proposition}
For a quasi-isometry class $[f]\in \QI(\mathbb{R}^n)$, the centralizer
\[
C_{\QI(\mathbb{R}^n)}([f])=\{[g]\in \QI(\mathbb{R}^n): [f]\circ[g]=[g]\circ[f]\}
\]
satisfies the following:
\begin{enumerate}
    \item $C_{\QI(\mathbb{R}^n)}([T_v]) = \QI(\mathbb{R}^n)$, where $T_v(x)=x+v$ is a translation.

    \item $\mathcal{A} \subset C_{\QI(\mathbb{R}^n)}([D_{\lambda}])$, where
  $\mathcal{A}=\{f:\mathbb{R}^n\to\mathbb{R}^n : f(x)=Ax+b,\; A\in \GL(n,\mathbb{R}),\, b\in\mathbb{R}^n\}$
    and $D_{\lambda}(x)=\lambda x$ for $\lambda\neq 1$. Moreover, this inclusion is proper.

 \end{enumerate}
\end{proposition}
\begin{proof}
\noindent (1)
Since $T_v(x)=x+v$ differs from the identity by a uniformly bounded amount, every class commutes with $[T_v]$. Therefore,
$C_{\QI(\mathbb{R}^n)}([T_v]) = \QI(\mathbb{R}^n).$

\noindent (2)
Let $f(x)=Ax+b\in\mathcal{A}$. Then
$(D_{\lambda}\circ f)(x)-(f\circ D_{\lambda})(x)
= \lambda(Ax+b) - (A\lambda x + b)
= (\lambda-1)b.$
The right-hand side is a constant vector; hence $\sup_{x\in\mathbb{R}^n}\|(D_{\lambda}\circ f)(x)-(f\circ D_{\lambda})(x)\| < \infty,$
which yields $[f]\in C_{\QI(\mathbb{R}^n)}([D_{\lambda}])$.

\noindent To see the inclusion is proper, consider in polar coordinates on $\mathbb{R}^2$ the map $f(r,\theta)=(e^{\sin\theta}r,\theta).$
This map is not affine, hence $f\notin \mathcal{A}$. However, $D_{\lambda}(r,\theta)=(\lambda r,\theta)$, and one can check that $D_{\lambda}\circ f = f\circ D_{\lambda}.$ Therefore, $[f]\in C_{\QI(\mathbb{R}^n)}([D_{\lambda}])$ while $f\notin\mathcal{A}$.
\end{proof}
\section{Asymptotic topology on $\QI(\mathbb{R}^n)$}\label{Discussion}

In the previous sections, we studied algebraic properties of $\QI(\mathbb{R}^n)$ via the subgroups $H$ and $H_\alpha$, defined through asymptotic deviation from the identity. As quasi-isometries are inherently large-scale objects, it is natural to consider a topology reflecting their behavior at infinity.

The compact-open topology is well suited for studying local properties of maps. However, quasi-isometries are inherently large-scale objects, and the subgroup $H$ is defined through asymptotic behavior at infinity. These considerations motivate the introduction of an asymptotic topology on $\QI(\mathbb{R}^n)$, defined in terms of relative deviation at infinity and naturally adapted to the large-scale structure of quasi-isometries and the subgroup $H$.
\subsection{Asymptotic topology}

Let $[f] \in \QI(\mathbb{R}^n)$, $\varepsilon > 0$, and $R > 0$. Define
\[
U([f]; \varepsilon, R)
=
\left\{
[g] \in \QI(\mathbb{R}^n)
:
\sup_{\|x\|\ge R} \frac{\|g(x)-f(x)\|}{\|x\|} < \varepsilon
\right\}.
\]
\begin{proposition}
The collection
\[
\mathcal{B}
=
\left\{
U([f]; \varepsilon, R)
:
[f] \in \QI(\mathbb{R}^n),\ \varepsilon>0,\ R>0
\right\}
\]
forms a basis for a topology on $\QI(\mathbb{R}^n)$.
\end{proposition}
\begin{proof}
For each $[f]$, clearly $[f] \in U([f]; \varepsilon, R)$ for all $\varepsilon, R > 0$.
Let $[h] \in U([f_1]; \varepsilon_1, R_1) \cap U([f_2]; \varepsilon_2, R_2)$. Set
\[
A_i = \sup_{\|x\|\ge R_i} \frac{\|h(x)-f_i(x)\|}{\|x\|}, \quad i=1,2.
\]
Choose $\varepsilon' = \min\{\varepsilon_1 - A_1,\ \varepsilon_2 - A_2\} > 0$ and $R' = \max\{R_1, R_2\}$. 

\noindent Then, for any $[g] \in U([h]; \varepsilon', R')$,
\begin{align*}
\sup_{\|x\|\ge R_i} \frac{\|g(x)-f_i(x)\|}{\|x\|}
&\le
\sup_{\|x\|\ge R'} \frac{\|g(x)-h(x)\|}{\|x\|} + A_i
< \varepsilon_i, \text{ which shows}\\
U([h]; \varepsilon', R') &\subset
U([f_1]; \varepsilon_1, R_1) \cap U([f_2]; \varepsilon_2, R_2).
\end{align*}
If \(f, g : \mathbb{R}^n \to \mathbb{R}^n\) are quasi-isometries such that
\([f] = [g]\), then there exists \(M > 0\) with
\(\displaystyle\sup_{x\in\Rn} \|f(x) - g(x)\| \le M\).
For any \([h] \in U([f]; \varepsilon, R)\),
\[
\|h(x) - g(x)\| \le \|h(x) - f(x)\| + \|f(x) - g(x)\|
\le \varepsilon \|x\| + M,
\]
so
\[
\sup_{\|x\| \ge R} \frac{\|h(x) - g(x)\|}{\|x\|}
\le \varepsilon + \frac{M}{R}.
\]
Thus, \(U([f]; \varepsilon, R) \subseteq U([g]; \varepsilon + M/R, R)\),
and, by symmetry, the reverse inclusion holds as well. Therefore, the family of basic neighborhoods is independent of the choice of representatives, and the topology is well defined on $\QI(\mathbb{R}^n)$.
\end{proof}
\begin{definition}
The topology generated by the basis $\mathcal{B}$ is called the \emph{asymptotic topology} on $\QI(\mathbb{R}^n)$.
\end{definition}
\subsection[Asymptotic pseudo-metric and relation with H]{Asymptotic pseudo-metric and relation with \texorpdfstring{$H$}{H}}

The asymptotic topology admits a natural description in terms of a pseudo-metric that captures large-scale deviation between quasi-isometries. Define
\[
d([f],[g]) = \limsup_{\|x\|\to\infty} \frac{\|f(x)-g(x)\|}{\|x\|}.
\]
It is straightforward to verify that $d$ is a well-defined pseudo-metric on $\QI(\mathbb{R}^n)$.

Moreover, $d$ induces a genuine metric on the quotient $\QI(\mathbb{R}^n)/H$ via
\[
\bar d([f]H,[g]H) := d([f],[g]),
\]
which is well defined and makes the quotient space Hausdorff. Thus, while $\QI(\mathbb{R}^n)$ itself carries only a pseudo-metric, passing to the quotient by $H$ yields a natural metric structure reflecting nontrivial asymptotic behavior. This provides a direct topological interpretation of the quotient studied in Section \ref{Properties}.

\subsection{An asymptotic stretch invariant}

The asymptotic topology also gives rise to natural invariants that encode the large-scale behavior of quasi-isometries. One such invariant measures the asymptotic radial stretch of a quasi-isometry.
\begin{definition}
For $[f]\in \QI(\mathbb{R}^n)$, define its \emph{asymptotic stretch} by
\[
s([f])
=
\lim_{R\to\infty}
\sup_{\|x\|\ge R}
\frac{\|f(x)\|}{\|x\|}.
\]
\end{definition}
\begin{proposition}
The map
$s:\QI(\mathbb{R}^n)\longrightarrow\mathbb{R}_{>0}$
is well defined and continuous with respect to the asymptotic topology. Moreover, it descends to a continuous map
\[
\bar s:\QI(\mathbb R^n)/H\longrightarrow\mathbb R_{>0},
\]
defined by
\[
\bar s([f]H)=s([f]).
\]
In particular, \(s([h])=1\) for every \([h]\in H\).
\end{proposition}
\begin{proof}
Suppose $[f]=[g]$. Then there exists $C>0$ such that
\[
\|f(x)-g(x)\|\le C
\qquad\text{for all }x\in\mathbb{R}^n.
\]
By the reverse triangle inequality,
\[
\left|
\frac{\|f(x)\|}{\|x\|}
-
\frac{\|g(x)\|}{\|x\|}
\right|
\le
\frac{C}{\|x\|}.
\]
Taking the supremum over $\|x\|\ge R$ and letting $R\to\infty$ shows that
$s([f])=s([g])$. Hence $s$ is well defined.

Now let $[f_k]\to[f]$ in the asymptotic topology. Since this topology is induced by the asymptotic pseudo-metric $d$, we have
\[
d([f_k],[f])\longrightarrow0.
\]
Again using the reverse triangle inequality,
\[
\bigl|\|f_k(x)\|-\|f(x)\|\bigr|
\le
\|f_k(x)-f(x)\|,
\]
and therefore
\[
|s([f_k])-s([f])|
\le
d([f_k],[f]).
\]
Thus, $s([f_k])\to s([f])$, proving continuity.

Finally, if $[h]\in H$, then
\[
\frac{\|h(x)-x\|}{\|x\|}\longrightarrow0,
\]
and therefore
\[
\frac{\|h(x)\|}{\|x\|}
\le
1+\frac{\|h(x)-x\|}{\|x\|}
\longrightarrow1,
\]
while the reverse triangle inequality gives
\[
\frac{\|h(x)\|}{\|x\|}
\ge
1-\frac{\|h(x)-x\|}{\|x\|}
\longrightarrow1.
\]
Hence, $s([h])=1$.

Now let $[g]=[f][h]$ with $[h]\in H$. Since $f$ is a quasi-isometry, there exist constants
$L\ge1$ and $A\ge0$ such that
\[
\|f(u)-f(v)\|
\le
L\|u-v\|+A
\qquad
(u,v\in\mathbb{R}^n).
\]
Taking $u=h(x)$ and $v=x$, we obtain
\[
\frac{\|g(x)-f(x)\|}{\|x\|}
=
\frac{\|f(h(x))-f(x)\|}{\|x\|}
\le
L\frac{\|h(x)-x\|}{\|x\|}
+\frac{A}{\|x\|}
\longrightarrow0.
\]
Hence,
\[
\left|
\frac{\|g(x)\|}{\|x\|}
-
\frac{\|f(x)\|}{\|x\|}
\right|
\le
\frac{\|g(x)-f(x)\|}{\|x\|}
\longrightarrow0,
\]
which implies $s([g])=s([f])$. Therefore $s$ is constant on the cosets of $H$.
\end{proof}

\section*{Acknowledgments}
 The authors thank Dr. Prateep Chakraborty for his valuable suggestions and fruitful comments. The first author acknowledges financial support from the Indian Institute of Science Education and Research Bhopal, India, and the second author acknowledges financial support from IIT Palakkad, India.
\bibliographystyle{plain}
\bibliography{bibfile}

@article{ye2023group,
  title={The group of quasi-isometries of the real line cannot act effectively on the line},
  author={Ye, Shengkui and Zhao, Yanxin},
  journal={Algebraic \& Geometric Topology},
  volume={23},
  number={8},
  pages={3835--3847},
  year={2023},
  publisher={Mathematical Sciences Publishers}
}

@article{bhowmik_chakraborty2,
  title={A structure theorem and left-orderability of a quotient of quasi-isometry group of the real line},
  author={Bhowmik, Swarup and Chakraborty, Prateep},
  journal={Geometriae Dedicata},
  volume={218},
  number={1},
  pages={12},
  year={2024},
  publisher={Springer}
}

@article{bhowmik_chakraborty,
  title={A Combinatorial Criterion and Center for the quasi-isometry groups of Euclidean spaces},
  author={Bhowmik, Swarup and Chakraborty, Prateep},
  journal={Topology and its Applications},
  volume={342},
  pages={108795},
  year={2024},
  publisher={Elsevier}
}

@article{mitra_sankaran,
  title={Embedding certain diffeomorphism groups in the quasi-isometry groups of Euclidean spaces},
  author={Mitra, Oorna and Sankaran, Parameswaran},
  journal={Topology and its Applications},
  volume={265},
  pages={106833},
  year={2019},
  publisher={Elsevier}
}

@article{sankaran,
  title={On homeomorphisms and quasi-isometries of the real line},
  author={Sankaran, Parameswaran},
  journal={Proceedings of the American Mathematical Society},
  volume={134},
  number={7},
  pages={1875--1880},
  year={2006}
}

@article{Deroin_Navas_Rivas,
  title={Groups, orders, and dynamics},
  author={Deroin, Bertrand and Navas, Andr{\'e}s and Rivas, Crist{\'o}bal},
  journal={arXiv preprint arXiv:1408.5805},
  year={2014}
}

@article{Mann,
  title={Left-orderable groups that don’t act on the line},
  author={Mann, Kathryn},
  journal={Mathematische Zeitschrift},
  volume={280},
  number={3},
  pages={905--918},
  year={2015},
  publisher={Springer}
}

@article{chakraborty,
  title={On the center of the group of quasi-isometries of the real line},
  author={Chakraborty, Prateep},
  journal={Indian Journal of Pure and Applied Mathematics},
  volume={50},
  number={4},
  pages={877--881},
  year={2019},
  publisher={Springer}
}

@incollection{gromov,
  title={Rigidity of lattices: an introduction},
  author={Gromov, Mikhael and Pansu, Pierre},
  booktitle={Geometric Topology: Recent Developments: Lectures given on the 1st Session of the Centro Internazionale Matematico Estivo (CIME) held at Montecatini Terme, Italy, June 4--12, 1990},
  pages={39--137},
  year={2006},
  publisher={Springer}
}

@phdthesis{zhaothesis,
  title={The group of quasi-Isometries of the real Line},
  author={Zhao, Yanxin},
  year={2024},
  school={University of Liverpool}
}

@article{Farb,
  title={The quasi-isometry classification of lattices in semisimple Lie groups},
  author={Farb, Benson},
  journal={Mathematical research letters},
  volume={4},
  number={5},
  pages={705--717},
  year={1997},
  publisher={International Press of Boston}
}

@article{Farb_Mosher,
  title={A rigidity theorem for the solvable Baumslag-Solitar groups},
  author={Farb, Benson and Mosher, Lee},
  journal={Inventiones mathematicae},
  volume={131},
  number={2},
  pages={419--451},
  year={1998},
  publisher={Springer}
}

@article{Whyte,
  title={The large scale geometry of the higher Baumslag-Solitar groups},
  author={Whyte, K},
  journal={Geometric \& Functional Analysis GAFA},
  volume={11},
  number={6},
  pages={1327--1343},
  year={2001},
  publisher={Springer}
}

@book{loh,
  title={Geometric group theory},
  author={L{\"o}h, Clara},
  year={2017},
  publisher={Springer}
}
\end{document}